\documentclass[11pt]{amsart}
\usepackage{txfonts}
\usepackage{amscd,amssymb}
\usepackage{graphicx}
\usepackage[colorlinks,plainpages,backref,urlcolor=blue]{hyperref}
\usepackage{tikz}
\usetikzlibrary{cd, calc, arrows}
\usepackage{enumitem}
\usepackage{orcidlink}

\numberwithin{equation}{section}

\newtheorem{theorem}{Theorem}[section]
\newtheorem{corollary}[theorem]{Corollary}
\newtheorem{lemma}[theorem]{Lemma}
\newtheorem{prop}[theorem]{Proposition}

\theoremstyle{definition}
\newtheorem{definition}[theorem]{Definition}
\newtheorem{example}[theorem]{Example}
\newtheorem{remark}[theorem]{Remark}
\newtheorem{question}[theorem]{Question}

\newcommand{\A}{\mathcal{A}}
\newcommand{\B}{\mathcal{B}}
\newcommand{\VV}{\mathcal{V}}
\newcommand{\RR}{\mathcal{R}}
\newcommand{\XX}{\mathcal{X}}
\newcommand{\cN}{\mathcal{N}}
\newcommand{\DD}{\mathcal{D}}

\newcommand{\C}{\mathbb{C}}
\newcommand{\Z}{\mathbb{Z}}

\newcommand{\Q}{\mathbb{Q}}
\newcommand{\N}{\mathbb{N}}

\newcommand{\FF}{\mathbb{F}}
\newcommand{\CP}{\mathbb{CP}}
\renewcommand{\P}{\mathbb{P}}
\newcommand{\bo}{\mathbf{1}}

\newcommand{\ii}{\mathrm{i}}

\newcommand{\codim}{\operatorname{codim}}

\newcommand{\depth}{\operatorname{depth}}

\newcommand{\apl}{A_{\rm PL}}

\DeclareMathOperator{\GL}{GL}
\DeclareMathOperator{\rank}{rank}
\DeclareMathOperator{\Ch}{Ch}
\DeclareMathOperator{\TC}{TC}
\DeclareMathOperator{\ab}{ab}

\DeclareMathOperator{\Hom}{Hom}
\DeclareMathOperator{\Conf}{Conf}
\DeclareMathOperator{\spn}{span}

\renewcommand{\k}{\Bbbk}

\newcommand{\inj}{\hookrightarrow}
\newcommand\isom{\xrightarrow{
 \,\smash{\raisebox{-0.6ex}{\ensuremath{\scriptstyle\simeq}}}\,}}

\begin{document}

\title[Arrangements with non-formal Milnor fibers]%
{Hyperplane arrangements with non-formal \\ Milnor fibers}

\author[Alexander~I.~Suciu]{Alexander~I.~Suciu\,\orcidlink{0000-0002-5060-7754}}
\address{Department of Mathematics,
Northeastern University,
Boston, MA 02115, USA}
\email{\href{mailto:a.suciu@northeastern.edu}{a.suciu@northeastern.edu}}
\urladdr{\href{https://suciu.sites.northeastern.edu}%
{suciu.sites.northeastern.edu}}
\thanks{Partially supported by the project ``Singularities and Applications''
--- CF 132/31.07.2023 funded by the European Union --- NextGenerationEU ---
through Romania's National Recovery and Resilience Plan.}

\dedicatory{Dedicated to Mike Davis on the occasion of his 75th birthday}



\begin{abstract}
Each complex hyperplane arrangement $\A$ gives rise to a 
Milnor fibration of its complement. Building on work of Zuber, 
we give a combinatorial sufficient condition for the Milnor fiber 
$F(\A)$ to be non-$1$-formal, expressed in terms of the multinet 
structure on $\A$, and use it to produce an infinite family of 
monomial arrangements $\A(3k,3k,3)$ with non-formal Milnor fibers. 
We also review the relevant background on cohomology jump loci, 
formality, and the topology of Milnor fibers of arrangements. 
\end{abstract}

\maketitle

\section{Introduction}
\label{sec:intro}

Let $\A$ be a central arrangement of hyperplanes in $\C^{\ell}$,
with complement $M=M(\A)$, projectivized complement $U=U(\A)$,
and Milnor fiber $F=F(\A)$. It is well known that the complement
$M$ is formal \cite{Br}. Like $M$, the Milnor fiber $F$ is a smooth
affine variety, but Morgan's criterion for $1$-formality
\cite{Morgan}, namely the vanishing of $W_1H^1(F;\C)$, need not
hold: here $W_1H^1(F;\C)=H^1(F;\C)_{\ne 1}$, which is non-zero
whenever the monodromy acts non-trivially on $H^1(F;\C)$. Thus,
unlike $M$, the Milnor fiber lies outside the range where mixed
Hodge theory forces formality. This raises the question,
first posed explicitly in \cite{PS-formal}:

\begin{question}[\cite{PS-formal}]
\label{q:milnorf}
Is the Milnor fiber $F(\A)$ of a hyperplane arrangement always
$1$-formal?
\end{question}

Zuber \cite{Zu} gave the first negative answer, showing that $F$
is not $1$-formal for the Ceva arrangement $\A(3,3,3)$.
In this paper, we give a combinatorial sufficient condition for
the non-$1$-formality of $F(\A)$, in terms of the multinet
structure on $\A$, and use it to produce an infinite family
of arrangements with non-formal Milnor fibers.

\begin{theorem}
\label{thm:general}
Let $\A$ be a central arrangement of $n$ hyperplanes in $\C^\ell$
supporting at least two distinct reduced $3$-multinets, at least
one of which has a base point of multiplicity one.
Then the Milnor fiber $F(\A)$ is not $1$-formal.
\end{theorem}

Here a \emph{base point of multiplicity one} of a reduced multinet is
a flat in its base locus contained in exactly one hyperplane from
each class (see Definition~\ref{def:multinet}); every base point of
a $3$-net has this property. In particular, we obtain the following
corollary.

\begin{corollary}
\label{cor:two-nets}
If $\A$ supports at least two distinct $3$-nets, then the Milnor
fiber $F(\A)$ is not $1$-formal.
\end{corollary}

The proof of Theorem~\ref{thm:general} proceeds as follows.
The Milnor fiber $F$ is a regular $\Z_n$-cover of the
projectivized complement $U=U(\A)$, classified by the
\emph{diagonal character} $\rho_n\colon\pi_1(U)\to\Z_n$,
which sends each meridional generator $\bar\gamma_H$ to $1$.
Embedding $\Z_n\inj \C^*$ via $1\mapsto e^{2\pi \ii/n}$, we
view $\rho_n$ as a rank-one complex character; its position
relative to the characteristic varieties of $U$ controls
the Betti numbers of $F$ and the characteristic polynomial
of the algebraic monodromy.

The argument combines four ingredients. The first is the fact that
the \emph{reduced diagonal character} $\bar\rho_n=\rho_n^{n/3}$ lies
on the component of $\VV^1_1(U)$ associated to any reduced
$3$-multinet \cite[Thm.~8.3]{PS-plms17}; it is thus a torsion point
linking the components associated to the two reduced $3$-multinets
to the Milnor fiber cover (Lemma~\ref{lem:rho-on-torus}).
The remaining three are:
the pencils associated to the multinets and their lifts to
the Milnor fiber, following Zuber \cite{Zu}; the depth
increase at intersections of characteristic variety components
\cite{ACM}; and a mixed Hodge structure argument exploiting
the weight filtration on the monodromy eigenspaces.

For the Ceva arrangement $\A(3,3,3)$, which supports four distinct
$3$-nets, Corollary~\ref{cor:two-nets} recovers Zuber's result.
For the monomial arrangements $\A(3k,3k,3)$ with $k\ge 2$, more work
is needed: these arrangements have flats of multiplicity $3k$, so the
results of \cite{PS-plms17} relating the Aomoto--Betti number
$\beta_3(\A)$ to the existence of $3$-nets do not apply. Instead, we
pull back the four $3$-nets on $\A(3,3,3)$ along the branched cover
$[z_1:z_2:z_3]\mapsto[z_1^k:z_2^k:z_3^k]$ of $\CP^2$; this produces
four distinct reduced $3$-multinets on $\A(3k,3k,3)$, three of which
are not nets (Lemma~\ref{lem:pullback-nets}). Theorem~\ref{thm:general}
then yields the following result.

\begin{theorem}
\label{thm:monomial}
For every integer $k\ge 1$, the Milnor fiber of the
monomial arrangement $\A(3k,3k,3)$ is not $1$-formal.
\end{theorem}

Along the way, Lemma~\ref{lem:pullback-nets} also corrects an
assertion from \cite{CS99} regarding the resonance varieties of
these arrangements; see Remark~\ref{rem:cs99}.

The remainder of this paper is organized as follows.
\begin{itemize}[itemsep=2pt, leftmargin=2.6em, labelsep=0.5em]
\item[\S\ref{sec:cjl}] reviews cohomology jump loci ---
characteristic and resonance varieties and the Tangent Cone
Theorem --- together with the notions of $1$-formality and
partial formality that underlie our argument.
\item[\S\ref{sec:arrangements}] covers background on hyperplane
arrangements: complements, intersection lattices, multinets,
and pencils.
\item[\S\ref{sect:mf-arr}] discusses the Milnor fibration, the
mixed Hodge structure on $F(\A)$, the formula for the characteristic 
polynomial of the algebraic monodromy, 
as well as the Aomoto--Betti numbers $\beta_p(\A)$ 
and their relation to essential components of $\VV^1_1(U)$.
\item[\S\ref{sect:reflection-arrs}] specializes to finite
reflection groups, with emphasis on the monomial arrangements
$\A(m,m,3)$ and the Hessian arrangement $\A(G_{25})$.
\item[\S\ref{sect:main}] contains the proofs of
Theorems~\ref{thm:general} and~\ref{thm:monomial}, 
together with a discussion of the role of weight purity and the
limits of the pincer argument.
\item[\S\ref{sect:questions}] gives a $\beta_3$ criterion for 
non-$1$-formality of $F(\A)$ and closes with several open
questions, including whether non-$1$-formality can occur for
arrangements with $\beta_3(\A)\le 1$, whether the Milnor fibers
of the braid arrangements $\A_n$ are formal, whether the reduced
multinet hypothesis in Theorem~\ref{thm:general} can be further
relaxed, and whether the Milnor fiber of $\A(G_{25})$ --- 
conjecturally the only arrangement supporting a
non-trivial $4$-net --- is $1$-formal. It ends with a discussion
of strong formality in the sense of Kohno and Pajitnov.
\end{itemize}

\section{Cohomology jump loci and formality}
\label{sec:cjl}

\subsection{Characteristic and resonance varieties}
\label{subsec:cjl-def}

Let $X$ be a connected CW-complex of finite type, and set
$G=\pi_1(X)$; thus $G_{\ab}=H_1(X;\Z)$. The \emph{character group} of $X$
is $\Ch(X)=\Hom(G_{\ab},\C^*)$; this is an abelian, complex algebraic
group whose tangent space at the identity $\bo$ we identify with
$H^1(X;\C)$.

The \emph{characteristic varieties} of $X$ are the jump loci
for cohomology with rank-$1$ coefficients:
\[
  \VV^i_s(X)=\bigl\{\rho\in\Ch(X)\mid
  \dim H^i(X;\C_\rho)\ge s\bigr\},
\]
where $\C_\rho$ is $\C$ viewed as a $\C[G_{\ab}]$-module
via $g\cdot z=\rho(g)z$. The \emph{resonance varieties} of $X$
are the jump loci for the Aomoto complex: setting
$A=H^*(X;\C)$, for each $a\in A^1$ one forms the cochain
complex $(A,\delta_a)$ with differential $\delta_a(u)=a\cup u$,
and defines
\[
  \RR^i_s(X)=\bigl\{a\in H^1(X;\C)\mid
  \dim H^i(A,\delta_a)\ge s\bigr\}.
\]
Each $\RR^i_s(X)$ is a Zariski-closed, homogeneous subvariety
of $H^1(X;\C)$. The relationship between the two families of
jump loci is discussed in \S\ref{subsec:formality}. In degree $1$, 
both $\VV^1_s(X)$ and $\RR^1_s(X)$ depend only on $\pi_1(X)$. 

\subsection{Quasi-projective varieties}
\label{subsec:qp}

When $X$ is a smooth, quasi-projective variety, the cohomology
jump loci are severely constrained by the underlying algebraic
geometry. The definitive structural result, building on foundational
work of Arapura \cite{Ar} and Dimca--Papadima \cite{DP-ccm},
is the following.

\begin{theorem}[Budur--Wang {\cite{BW}}]
\label{thm:bw}
Let $X$ be a smooth quasi-projective variety. Then each
characteristic variety $\VV^i_s(X)$ is a finite union of
torsion-translated subtori of $\Ch(X)$.
\end{theorem}

The non-translated subtori occurring in $\VV^1_s(X)$ are
accounted for by Arapura's theorem. Call a holomorphic map
$f\colon X\to\Sigma$ \emph{admissible} if it is surjective with
connected generic fiber and the target $\Sigma$ is a smooth
connected complex curve with $\chi(\Sigma)<0$. Up to
reparametrization at the target, $X$ admits only finitely many
admissible maps; let $\mathcal{E}_X$ denote the set of
equivalence classes.

\begin{theorem}[Arapura {\cite{Ar}}]
\label{thm:arapura}
The map $f\mapsto f^*(\Ch(\Sigma))$ establishes a bijection
between $\mathcal{E}_X$ and the set of positive-dimensional
irreducible components of $\VV^1_1(X)$ containing the identity.
\end{theorem}

The positive-dimensional components of $\VV^1_1(X)$ that do
\emph{not} pass through the identity are described similarly,
by replacing admissible maps with orbifold fibrations $f\colon
X\to(\Sigma,\mathbf{m})$ to curves with multiple fibers; see
\cite{ACM,Su-imrn} for details.

The depth theorem of Artal Bartolo, Cogolludo, and Matei
\cite{ACM} controls how torsion characters interact with
higher-depth characteristic varieties.

\begin{theorem}[{\cite[Prop.~6.9]{ACM}}]
\label{thm:acm}
Let $X$ be a smooth quasi-projective variety. Suppose $T_1$
and $T_2$ are distinct positive-dimensional irreducible
components of $\VV^1_r(X)$ and $\VV^1_s(X)$, respectively.
If $\xi\in T_1\cap T_2$ is a torsion character, then
$\xi\in\VV^1_{r+s}(X)$.
\end{theorem}

\subsection{Formality and the Tangent Cone Theorem}
\label{subsec:formality}

A connected CW-complex $X$ of finite type is \emph{formal}
(over $\Q$) in the sense of Sullivan \cite{Sullivan77} if its
algebra of PL differential forms $(\apl(X),d)$ is connected by
a zig-zag of quasi-isomorphisms of cdgas to $(H^*(X;\Q),0)$.
More generally, $X$ is \emph{$q$-formal} if the zig-zag is only
required to induce isomorphisms in cohomology through degree $q$
and a monomorphism in degree $q+1$. Clearly, formal implies
$q$-formal for all $q$, and $1$-formality of $X$ depends only on
$\pi_1(X)$. For surveys of the formality notions needed for this paper,
we refer to \cite{PS-formal, Su-ems}.

The following result shows that for spaces of low topological
complexity, $q$-formality already implies full formality.

\begin{theorem}[M\u{a}cinic {\cite{Mc10}}]
\label{thm:partial-to-formal}
Let $X$ be a space with $H^i(X;\Q)=0$ for all $i\ge q+2$.
Then $X$ is $q$-formal if and only if $X$ is formal.
\end{theorem}

The interaction between formality and cohomology jump loci is
governed by two types of tangent cone associated to a subvariety
$W$ of the character group $\Ch(X)\cong H^1(X;\C^*)$. The first is
the classical \emph{tangent cone} $\TC_{\bo}(W)$ at the identity,
the homogeneous subvariety of $T_{\bo}\Ch(X)=H^1(X;\C)$ defined by
the initial forms of the local equations of $W$ at ${\bo}$. The
second is the \emph{exponential tangent cone} \cite{DPS-duke}
\begin{equation}
\label{eq:tau1}
  \tau_{\bo}(W) = \bigl\{z\in H^1(X;\C)\bigm|
  \exp(\lambda z)\in W\text{ for all }\lambda\in\C\bigr\}.
\end{equation}
Thus $\tau_{\bo}(W)$ consists of those directions along which 
the entire exponential line remains inside $W$.
As shown in \cite{DPS-duke}, $\tau_{\bo}(W)$ is always a finite union
of rationally defined linear subspaces of $H^1(X;\C)$, and the
inclusion $\tau_{\bo}(W)\subseteq\TC_{\bo}(W)$ 
holds for any $W$, but is in general strict. Applied to the
characteristic varieties, and combined with the tangent cone
inclusion of Libgober \cite{Li02}, one obtains for every $i,s$
the chain
\begin{equation}
\label{eq:tc-chain}
  \tau_{\bo}\bigl(\VV^i_s(X)\bigr)\subseteq
  \TC_{\bo}\bigl(\VV^i_s(X)\bigr)\subseteq\RR^i_s(X),
\end{equation}
where each inclusion may be strict in general.

The Tangent Cone Theorem asserts that formality forces all
inclusions in \eqref{eq:tc-chain} to be equalities.

\begin{theorem}[Tangent Cone Theorem {\cite{DPS-duke, DP-ccm}}]
\label{thm:tcone}
If $X$ is $q$-formal, then for all $i\le q$ and $s\ge 1$,
\[
  \tau_{\bo}\bigl(\VV^i_s(X)\bigr)=
  \TC_{\bo}\bigl(\VV^i_s(X)\bigr)=\RR^i_s(X).
\]
In particular, all irreducible components of $\RR^i_s(X)$ are
rationally defined linear subspaces of $H^1(X;\C)$, and the
irreducible components of $\VV^i_s(X)$ through the identity are
algebraic subtori of the form $\exp(L)$, where $L$ ranges over
the components of $\RR^i_s(X)$.
\end{theorem}

The contrapositive gives the key non-formality test: if either
inclusion in \eqref{eq:tc-chain} is strict for some $i\le q$,
then $X$ is not $q$-formal. This is the mechanism we use 
later to detect non-$1$-formality of the Milnor fiber.

For smooth quasi-projective varieties, an important source of
formality comes from mixed Hodge theory. Every smooth
quasi-projective variety $X$ has the homotopy type of a finite
CW-complex \cite{Di92}, and its rational cohomology carries a
mixed Hodge structure with weight filtration $W_\bullet$. By
Morgan's theorem \cite{Morgan}, if $W_1H^1(X;\C)=0$ then $X$
is $1$-formal. More generally, Dupont proved the following
\emph{purity implies formality} principle \cite[Thm.~1.1]{Dp16}:
if $H^k(X;\Q)$ is pure of weight $2k$ for all $k\le q$, then $X$
is $q$-formal. An analogous result for (possibly singular)
projective varieties, where purity of weight $k$ is required
instead, was proved by Chataur and Cirici \cite[Thm.~3.3]{CC}.

For $1$-formal spaces in general, the Tangent Cone Theorem
guarantees that the irreducible components of $\RR^1_1(X)$ are
rational linear subspaces, but they may intersect non-trivially:
this occurs already for right-angled Artin groups \cite{PS-mathann06}.
For smooth quasi-projective varieties, however, $1$-formality
implies the stronger conclusion that the components of
$\RR^1_1(X)$ are \emph{pairwise disjoint} outside the origin
\cite[Thm.~C]{DPS-duke}. The proof of Theorem~\ref{thm:general}
uses a version of this property that requires no formality
hypothesis: distinct positive-dimensional components of
$\VV^1_1(X)$ through $\bo$ meet in finitely many points
\cite[Lem.~6.12]{DPS-duke}, and so their tangent spaces at $\bo$
meet only at $0$; see~\eqref{eq:disjoint} and
Remark~\ref{rem:tc-strict}.

Finally, we record the following consequence of the Transfer
theorem for finite covers \cite[Cor.~2.5]{DP-pisa}, which
will be used several times below.

\begin{lemma}[{\cite{DP-pisa}}]
\label{lem:dp-formality}
Let $p\colon Y\to X$ be a finite regular cover with deck group
$H$. If $X$ is $q$-formal and $H$ acts trivially on
$H^i(Y;\Q)$ for all $i\le q$, then $Y$ is $q$-formal.
\end{lemma}

\section{Hyperplane arrangements}
\label{sec:arrangements}

\subsection{Complements and intersection lattices}
\label{subsec:arr}

A \emph{central hyperplane arrangement} $\A$ in $\C^{\ell}$ is a
finite collection of codimension-$1$ linear subspaces
$H\subset\C^{\ell}$. We write $n=|\A|$ for the number of
hyperplanes, and $Q=Q(\A)=\prod_{H\in\A}f_H$ for a defining
polynomial, where each $f_H$ is a linear form with $\ker(f_H)=H$.
The \emph{complement} is a connected, smooth quasi-projective variety
\[
  M = M(\A) = \C^{\ell}\setminus\bigcup_{H\in\A}H.
\]
Since $Q$ is homogeneous, $\C^*$ acts freely on $M$ by scalar
multiplication, and the orbit space is the \emph{projectivized
complement} $U=U(\A)=\P(M)\subset\CP^{\ell-1}$. The Hopf bundle
restricts to a trivial fibration $M(\A)\cong U(\A)\times\C^*$.

The \emph{intersection lattice} $L(\A)$ is the poset of all
non-empty intersections of hyperplanes in $\A$, ordered by
reverse inclusion and ranked by codimension. We write $L_k(\A)$
for the set of rank-$k$ flats, and $\mu\colon L(\A)\to\Z$ for
the M\"obius function, defined by $\mu(\C^\ell)=1$ and
$\mu(X)=-\sum_{Y\supsetneq X}\mu(Y)$. The arrangement is
\emph{essential} if $\bigcap_{H\in\A}H=\{0\}$, or equivalently if
$r=\ell$, where $r=\rank(\A)$. In general, there is a decomposition
\[
  M(\A)\cong M(\A')\times \C^{\ell-r},
\]
where $\A'$ is an essential arrangement in $\C^r$ (the
\emph{essentialization} of $\A$). Accordingly, most homotopy-theoretic
properties of $M(\A)$ and of the Milnor fiber depend only on the
essentialization.

For each $k\ge 1$, the Betti numbers of $M(\A)$ are given by
the combinatorial formula $b_k(M)=\sum_{X\in L_k(\A)}|\mu(X)|$,
and all homology groups of $M(\A)$ are torsion-free \cite{OS80, OT}.
The multiplicity of a flat $X\in L_2(\A)$ is
$q_X=|\A_X|$, where $\A_X=\{H\in\A:H\supset X\}$. 
In the rank-$3$ case $\ell=3$, a flat $X\in L_2(\A)$ corresponds 
to a $q_X$-fold multiple point $\P(X)\in\P(\A)$.

The product decomposition $M\cong U\times\C^*$ induces a splitting
$\pi_1(M)\cong \pi_1(U)\times \Z$, where the central factor
$\Z=\pi_1(\C^*)$ is generated by the product of the meridional
loops $\gamma_H$ about the hyperplanes $H\in\A$. Let
$\overline{\gamma}_H=\pi_\sharp(\gamma_H)\in \pi_1(U)$ denote the
image of $\gamma_H$ under the projection
$\pi_\sharp\colon \pi_1(M)\to\pi_1(U)$. Passing to homology yields
$H_1(M;\Z)\cong \Z^n$, with basis the classes $x_H=[\gamma_H]$,
and $H_1(U;\Z)\cong \Z^{n-1}$, generated by the classes
$\overline{x}_H=[\overline{\gamma}_H]$, subject to the single
relation $\sum_{H\in\A}\overline{x}_H=0$.

Since the fundamental groups of $M$ and $U$, as
well as their meridional bases, 
are preserved under passage to a generic $3$-dimensional slice,
by the Lefschetz-type theorem of Hamm--L\^{e} \cite{HL},
for most purposes it suffices to assume that $\rank(\A)=3$,
i.e., that $\A$ is a central arrangement in $\C^3$ and
$\P(\A)$ is a line arrangement in $\CP^2$.

\subsection{The Orlik--Solomon algebra and formality}
\label{subsec:os}

The cohomology ring of $M(\A)$ was computed by Brieskorn
\cite{Br}, building on Arnol'd's calculation for the braid
arrangement, and given a purely combinatorial description by
Orlik--Solomon \cite{OS80}. Fix a linear order on $\A$, let
$E=E(\A)$ be the exterior algebra over $\Z$ with generators
$e_H$ in degree~$1$, and define a degree-$(-1)$ derivation
$\partial\colon E\to E$ by $\partial(1)=0$, $\partial(e_H)=1$,
extended by the graded Leibniz rule. Then
\[
  H^*(M(\A);\Z)\cong A(\A)\coloneqq E(\A)/I(\A),
\]
where $I(\A)$ is generated by $\partial(e_\B)$ for all
$\B\subset\A$ with $\codim\bigl(\bigcap_{H\in\B}H\bigr)<|\B|$.
The \emph{Orlik--Solomon algebra} $A(\A)$ is torsion-free and
depends only on $L(\A)$; in particular, all Betti numbers of
$M(\A)$ are combinatorial invariants.

The complement $M(\A)$ is \emph{formal} in the sense of rational
homotopy theory. This follows from Dupont's \emph{purity implies
formality} theorem \cite[Thm.~1.1]{Dp16}, since $H^k(M(\A);\Q)$ is
pure of weight $2k$ for all $k$ \cite{Shapiro}; see also
\cite[Thm.~3.7]{Dp16}. Alternatively, for each $H\in\A$ the logarithmic
$1$-form $\omega_H=\tfrac{1}{2\pi \ii}\,d\log f_H$ restricts to a
closed form on $M$, and the subalgebra $\mathcal{D}\subset
\Omega^*_{\mathrm{dR}}(M;\C)$ generated by these forms maps
isomorphically onto $H^*(M;\C)$; this exhibits $M(\A)$ as formal
over $\C$, and hence over $\Q$ \cite{Sullivan77}.
In particular, all (rational) Massey products in
$H^*(M(\A);\Q)$ vanish, and the arrangement group
$G(\A)=\pi_1(M(\A))$ is $1$-formal. 

\subsection{Cohomology jump loci of arrangements}
\label{subsec:cv-arr}

We now describe the cohomology jump loci of the complement
$M=M(\A)$ and the projectivized complement $U=U(\A)$, and
the relationship between them. The resonance varieties 
$\RR^i_1(M(\A))$ were first defined and studied by Falk in \cite{Fa97}.
The Tangent Cone Theorem was first established in this setting, in
degree one, by Cohen and Suciu \cite[Thm.~5.2]{CS99}, who also
studied the first characteristic and resonance varieties of the
braid and monomial arrangements discussed below \cite[\S6]{CS99}.

Since $M(\A)\cong U(\A)\times\C^*$, the K\"unneth formula gives
$H^1(M;\C)=H^1(U;\C)\oplus H^1(\C^*;\C)$, with $b_1(M)=n$
and $b_1(U)=n-1$. The projection $\pi\colon M\to U$ and the
splitting $M\cong U\times\C^*$ interact with cohomology jump
loci as follows \cite[Prop.~3.3, Cor.~6.13]{Su-revroum}.

\begin{prop}
\label{prop:cjl-MU}
The map $\pi^*\colon H^1(U;\C)\inj H^1(M;\C)$ induces
isomorphisms
\[
  \RR^i_1(U)\isom\RR^i_1(M) \quad\text{and}\quad
  \VV^i_1(U)\isom\VV^i_1(M)
\]
for all $1\le i\le r-1$, where $r=\rank(\A)$. Moreover, 
$\RR^1_s(U)\cong\RR^1_s(M)$ and $\VV^1_s(U)\cong\VV^1_s(M)$ for 
all $1\le s<n$.
\end{prop}

Since $M$ is a smooth, quasi-projective variety, 
Theorem \ref{thm:bw} ensures that each characteristic variety 
$\VV^i_s(M)$ is a finite union of torsion-translated subtori of 
the character torus $H^1(M;\C^*)=(\C^*)^n$. Moreover, 
since $M$ is formal, Theorem \ref{thm:tcone} gives
$\TC_{\bo}(\VV^i_s(M))=\RR^i_s(M)$ for all $i,s$, and therefore 
each resonance variety $\RR^i_s(M)$ is a finite union of 
rationally defined linear subspaces of $H^1(M;\C)=\C^n$. 

As shown by Davis, Januszkiewicz, Leary, and Okun \cite{DJLO11},
every arrangement complement $M=M(\A)$ is a \emph{duality space}
in the sense of Bieri--Eckmann \cite{BE}. Motivated in part by
the quest for a topological explanation of propagation of
resonance varieties, Denham, Suciu, and Yuzvinsky \cite{DSY16,
DSY17} introduced the related notion of \emph{abelian duality
space} and proved that $M$ is an abelian duality space of
dimension $r=\rank(\A)$. Since $M$ is also formal, the Tangent
Cone Theorem transfers propagation of characteristic varieties
to resonance varieties:
\begin{equation}
\label{eq:propagation}
  \VV^1_1(M)\subseteq\cdots\subseteq\VV^{r}_1(M),\qquad
  \RR^1_1(M)\subseteq\cdots\subseteq\RR^{r}_1(M),
\end{equation}
and likewise $\VV^1_1(U)\subseteq\cdots\subseteq\VV^{r-1}_1(U)$
and $\RR^1_1(U)\subseteq\cdots\subseteq\RR^{r-1}_1(U)$; 
see \cite{Su-revroum} for a detailed account.

The variety $\VV^1_1(M)$ may also contain positive-dimensional, 
torsion-translated subtori \cite{Su02}. Such components do not pass
through the identity, and are thus not detected by resonance.
The ones described in \cite{DeS-plms} arise from certain 
``pointed multinets'' on the arrangement, but there are 
some that arise in other ways \cite{LX}.

\subsection{Multinets and resonance components}
\label{subsec:nets}

The combinatorial data underlying the first resonance variety of $M(\A)$
is encoded in the notion of a multinet.

\begin{definition}
\label{def:multinet}
A \emph{multinet} $\cN$ on $\A$ is a partition
$\A=\A_1\sqcup\cdots\sqcup\A_k$ ($k\ge 3$) together with a
multiplicity function $m\colon\A\to\N_{>0}$ and a subset
$\XX\subseteq L_2(\A)$ called the \emph{base locus}, such that:
\begin{enumerate}[label=\textup{(\roman*)}, itemsep=1pt]
  \item \label{m1}
  the sum $d=\sum_{H\in\A_i}m_H$ is independent of $i$;
  \item \label{m2}
  for $H\in\A_i$ and $H'\in\A_j$ with $i\ne j$, the flat
        $H\cap H'$ lies in $\XX$;
  \item \label{m3}
  for each $X\in\XX$, the sum
        $n_X=\sum_{\substack{H\in\A_i \\ H\supset X}}m_H$
        is independent of $i$;
  \item \label{m4}
       within each class $\A_i$, any two hyperplanes can be
        connected by a chain with consecutive members meeting
        outside $\XX$.
\end{enumerate}
The common value $d$ is the \emph{weight}; we call $\cN$ a 
\emph{$(k,d)$-multinet}, or simply a \emph{$k$-multinet}.
If all $m_H=1$, the multinet is \emph{reduced}; if in addition 
every flat in $\XX$ is contained in exactly one hyperplane from each class, 
the multinet is a \emph{$(k,d)$-net}, or simply a \emph{$k$-net}. 
A $k$-net is \emph{non-trivial} if $d>1$, equivalently if 
$\rank(\A)\ge 3$. 
\end{definition}

By work of Pereira--Yuzvinsky \cite{PY08} and
Yuzvinsky \cite{Yu09}, if $\cN$ is a multinet with $|\XX|>1$
then $k\le 4$; moreover $k=3$ if any multiplicity $m_H\ne 1$.
Nets with $k=3$ arise in infinite families; the only known non-trivial 
$(4,d)$-net is the $(4,3)$-net on the Hessian arrangement
(see \S\ref{subsec:hessian}).

Since the slicing operation of \S\ref{subsec:arr} does not change
the resonance varieties $\RR^1_s(M(\A))$ or the characteristic varieties
$\VV^1_s(M(\A))$, we may assume without loss of generality that $\ell=3$.

Following \cite[Thm.~3.11]{FY07}, building on \cite{LY00}, a
$(k,d)$-multinet $\cN$ on an arrangement $\A$ in $\C^3$ 
determines an admissible map
\[
  f_\cN\colon M(\A)\longrightarrow\Sigma_{0,k},
  \qquad f_\cN=[C_1:C_2],
\]
where $\Sigma_{0,k}=\CP^1\setminus\{k\text{ points}\}$ and
$C_i=\prod_{H\in\A_i}f_H^{m_H}$ for $i=1,\dots,k$. The
condition $\dim\spn\{C_1,\dots,C_k\}=2$ follows from the
multinet axioms, and the $k$ punctures of $\Sigma_{0,k}$
correspond to the values $[C_1],\dots,[C_k]$ in this pencil.
The induced map $f_\cN^*\colon H^1(\Sigma_{0,k};\C)\to
H^1(M(\A);\C)$ is injective; its image is the rational linear subspace
\[
  P_\cN = f_\cN^*(H^1(\Sigma_{0,k};\C))
  = \spn\{u_2-u_1,\dots,u_k-u_1\}
  \subset H^1(M(\A);\C)
\]
of dimension $k-1$, where $u_i=\sum_{H\in\A_i}m_H\,e_H$.
Since $H^2(\Sigma_{0,k};\C)=0$, the space $H^1(\Sigma_{0,k};\C)$ 
is isotropic under the cup product, and hence $P_\cN$ is also isotropic. 
As shown in \cite[Thm.~2.4]{FY07}, $P_\cN$ is an
irreducible component of $\RR^1_1(M(\A))$, and the corresponding
component of $\VV^1_1(M(\A))$ through the identity is
$T_\cN=\exp(P_\cN)$, by formality of $M(\A)$.

More generally, if $\B\subseteq\A$ is a sub-arrangement
supporting a multinet $\cN$, the inclusion $M(\A)\hookrightarrow
M(\B)$ induces an injection $H^1(M(\B);\C)\hookrightarrow
H^1(M(\A);\C)$, and the image of $P_\cN$ under this embedding
is again an irreducible component of $\RR^1_1(M(\A))$, which we
continue to denote $P_\cN$. Conversely, by \cite[Thm.~2.4]{FY07},
every positive-dimensional irreducible component of
$\RR^1_1(M(\A))$ is of this form, for some multinet $\cN$ on
some sub-arrangement $\B\subseteq\A$. The simplest examples are
the \emph{local components}: for each flat $X\in L_2(\A)$ with
$q_X\ge 3$, the localization $\A_X$ carries a trivial
$(q_X,1)$-net, contributing a $(q_X-1)$-dimensional component
$L_X$. Components arising from multinets on the whole of $\A$
are called \emph{essential}; it is these components --- and the
pencils $f_\cN$ associated to them --- that play the central role
in the Milnor fiber story developed in \S\ref{sect:mf-arr}.

Finally, the higher-depth resonance varieties $\RR^1_s(M(\A))$
for $s\ge 2$ require no new ingredients: since the components of
$\RR^1_1(M(\A))$ are linear subspaces meeting only at the
origin, the depth-$s$ variety is simply the union of those
components of $\RR^1_1(M(\A))$ having dimension 
$\ge s+1$ \cite{LY00}. 

\section{Milnor fibrations of arrangements}
\label{sect:mf-arr}

\subsection{The Milnor fibration}
\label{subsec:mf-background}

Let $\A$ be a central arrangement of $n$ hyperplanes in $\C^{\ell}$,
of rank $r=\rank(\A)$, with defining polynomial $Q=Q(\A)$.
Since $Q(\A)$ is homogeneous of degree $n=|\A|$,
Milnor's fibration theorem \cite{Milnor} gives a smooth fiber bundle
\[
  Q\colon M(\A)\longrightarrow\C^*.
\]
Its generic fiber, the \emph{Milnor fiber} $F=F(\A)=Q^{-1}(1)$,
is a smooth affine algebraic variety of complex dimension $\ell-1$.
As such, $F$ is a Stein manifold, and by a classical theorem of
Andreotti--Frankel, it has the homotopy type of a finite
CW-complex of dimension at most $\ell-1$ (in fact, at most $r-1$).
The \emph{geometric monodromy} $h\colon F\to F$,
$h(z)=e^{2\pi \ii/n}z$, has finite order $n$ and generates the deck
transformation group of the cyclic cover
\[
  \sigma\colon F(\A)\longrightarrow U(\A),
\]
classified by the \emph{diagonal character}
\[
  \rho_n\colon\pi_1(U(\A))\longrightarrow\Z_n,
  \qquad\bar{\gamma}_H\longmapsto 1\text{ for all }H\in\A;
\]
see \cite{CS-jlms95}. 
Since the slicing operation of \S\ref{subsec:arr} preserves
$\pi_1(U(\A))$ and its meridional basis, it also preserves the
cover $\sigma\colon F(\A)\to U(\A)$, and hence $\pi_1(F(\A))$.

The monodromy $h$ induces a decomposition
\begin{equation}
\label{eq:H1F}
  H^1(F;\C)=\bigoplus_{j=0}^{n-1}H^1(F;\C)_{\theta^j},
  \quad\theta=e^{2\pi \ii/n},
\end{equation}
with eigenvalue-$1$ subspace $H^1(F;\C)_1\cong H^1(U;\C)$.
By \cite[Thm.~4.1]{DP-pisa} (applied, if $\ell>3$, to a generic
$3$-dimensional slice), $F$ carries a Deligne mixed Hodge structure
in which
\begin{align*}
  H^1(F;\C)_1 &= H^{1,1}(F), &\text{(pure of weight $2$)},\\
  W_1(F) &= H^1(F;\C)_{\ne 1}, &\text{(pure of weight $1$)}. 
\end{align*}
Moreover, the weight-$1$ piece decomposes as 
$W_1(F)=H^{1,0}(F)\oplus H^{0,1}(F)$, with  
$\dim H^{1,0}(F)=\dim H^{0,1}(F)$.

\subsection{Betti numbers and algebraic monodromy}
\label{subsec:mf-betti}

Fix the embedding $\Z_n \inj \C^*$ given by $1\mapsto
\theta=e^{2\pi \ii/n}$. Via this identification, we view
$\rho_n$ and its powers as rank-one complex characters in
$\Ch(U(\A))=\Hom(\pi_1(U(\A)),\C^*)\cong (\C^*)^{n-1}$.
The position of $\rho_n$ in the characteristic varieties of
$U(\A)$ controls $b_1(F(\A))$ and the characteristic polynomial
of the algebraic monodromy, given by
\begin{equation}
\label{eq:charpoly}
  \Delta_1(t) = (t-1)^{n-1}
  \prod_{\substack{j\mid n \\ j>1}}\Phi_j(t)^{e_j(\A)},
  \qquad e_j(\A) = \depth_1(\rho_n^{n/j}),
\end{equation}
where $\depth_1(\rho)=\max\{s\ge 0\mid \rho\in\VV^1_s(U)\}$
and $\Phi_j$ is the $j$-th cyclotomic polynomial
\cite{DeS-plms, Su-toul}.

Only essential components of $\VV^1_1(U)$ contribute to
$\Delta_1$. Here, in the coordinates $t_H=\rho(\bar\gamma_H)$ on
$\Ch(U)=\{t\in(\C^*)^n\mid \prod_{H\in\A} t_H=1\}$, a subvariety is
\emph{essential} if it is not contained in any of the coordinate
subtori $\{t_H=1 \text{ for all } H\notin\B\}$ with $\B\subsetneq\A$.
Each power $\rho_n^j\ne\bo$ has all coordinates equal to
$\theta^j\ne 1$, and so lies in no such subtorus. Hence any
non-essential component of $\VV^1_1(U)$ --- in particular any
local component or any component arising from a multinet on a
proper sub-arrangement of $\A$ --- contains no non-trivial power
of $\rho_n$, and therefore contributes nothing to the exponents
$e_j(\A)$ in~\eqref{eq:charpoly}.

Among the essential components, it is a further and more subtle
fact that the algebraic monodromy is controlled by whether the
supporting multinets are reduced:

\begin{prop}[{\cite[Thm.~8.3]{PS-plms17}}]
\label{prop:red-multi}
Let $\A$ be a central arrangement in $\C^3$.
If $\A$ admits a reduced multinet, then the algebraic
monodromy on $H_1(F;\C)$ is non-trivial.
\end{prop}

The reducedness hypothesis cannot be dropped: if $\A$ supports
essential multinets but none is reduced, the monodromy may
still be trivial. This is illustrated by the
$\operatorname{B}_3$ reflection arrangement, whose single
essential component arises from a non-reduced $(3,4)$-multinet
and contains no non-trivial power of $\rho_n$,
giving $\Delta_1(t)=(t-1)^8$; see \cite[\S 10.1]{Su-revroum}.
Conversely, the monodromy may be non-trivial even
without a reduced multinet, as illustrated by the arrangements
of type $G(3d+1,1,3)$ with $d>0$ \cite[Ex.~8.11]{PS-plms17}.
The precise mechanism by which a reduced $3$-multinet $\cN$ on $\A$
forces $\bar\rho_n\in T_\cN$ is the content of
Lemma~\ref{lem:rho-on-torus} below.

\subsection{Pencils of lines in the plane}
\label{subsec:pencil}

Let $\A$ be a pencil of $n\ge 3$ lines through the origin
of $\C^2$, defined, say, by $Q=x^n-y^n$; the defining polynomial
has an isolated singularity at $0$.
The projectivized complement is $U(\A)=\Sigma_{0,n}$,
a sphere with $n$ punctures, and $M(\A)\cong U(\A)\times\C^*$.
The Milnor fiber $F(\A)=Q^{-1}(1)=\Sigma_{g,n}$ is a surface
of genus $g=\tfrac{(n-1)(n-2)}{2}$ with $n$ punctures,
giving $b_1(F)=(n-1)^2$, consistent with
$\Delta_1(t)=(t-1)^{n-1}\prod_{j\mid n,\, j>1}\Phi_j(t)^{n-2}$
from formula~\eqref{eq:charpoly}, since $\VV^1_1(U)=\Ch(U)\cong(\C^*)^{n-1}$,
and $\depth_1(\rho)=-\chi(U)=n-2$ for every $\rho\ne\bo$.
This arrangement supports a \emph{trivial} $(n,1)$-net, with one
hyperplane in each class and base locus consisting of
the single flat $\{0\}$; the
essential component $\VV^1_1(U)=\Ch(U)$ arises from
this net.

The case $n=3$ is the building block for the entire argument.
Here $F(\A)=\Sigma_{1,3}$, a thrice-punctured elliptic curve,
with $b_1(F)=4$ and $\Delta_1(t)=(t-1)^2(t^2+t+1)$.
The monodromy has order $3$ and acts on $H^1(F;\C)$
with eigenvalue-$1$ subspace $H^1(F;\C)_1\cong H^1(U;\C)=\C^2$
and weight-$1$ piece
\[
W_1(F) = H^{1,0}(F)\oplus H^{0,1}(F), \qquad 
\dim H^{1,0}(F)=\dim H^{0,1}(F)=1,
\]
by \cite[Thm.~4.1]{DP-pisa}. The Milnor fiber cover
$\sigma\colon F(\A)\to U(\A)$ is exactly the
$\Z_3$-cover classified by the diagonal character
$\rho_3\colon\pi_1(U)\to\Z_3$, where
$\pi_1(U)\cong F_2$ is the free group of rank $2$.

\subsection{The reduced diagonal character}
\label{subsec:pencil-rho}

When $\A$ supports a reduced $3$-multinet, we have $3\mid n$
by the definition of multinet. The \emph{reduced diagonal character} is
\begin{equation}
\label{eq:bar-rho}
  \bar\rho_n \coloneqq \rho_n^{n/3},
\end{equation}
a primitive character of order $3$, sending each
$\bar\gamma_H\mapsto e^{2\pi \ii/3}$. 
The braid arrangement illustrates why one must work
with $\bar\rho_n$ rather than $\rho_n$.

\begin{example}
\label{ex:braid-rho}
Let $\A=\A_3$ be the braid arrangement in $\C^3$, defined by
$Q=(x+y)(x-y)(x+z)(x-z)(y+z)(y-z)$.
Labeling the hyperplanes as the factors of $Q$, the triple
points are $136$, $145$, $235$, $246$.
The arrangement supports a $(3,2)$-net with partition
$(12|34|56)$, which defines a pencil
$\psi\colon U\to\Sigma_{0,3}$,
$[x:y:z]\mapsto[x^2-y^2:x^2-z^2]$,
yielding by pullback a $2$-dimensional essential component
of $\VV^1_1(U)$:
\[
T=\{(s,s,t,t,(st)^{-1},(st)^{-1})\mid s,t\in\C^*\}
\subset\Ch(U)\cong(\C^*)^5,
\]
in coordinates $t_H=\rho(\bar\gamma_H)$ subject to $\prod_H t_H=1$.
The diagonal character $\rho_6\colon\bar\gamma_H\mapsto
e^{2\pi \ii/6}$ satisfies $\rho_6\notin T$, while
$\bar\rho_6=\rho_6^2\colon\bar\gamma_H\mapsto e^{2\pi \ii/3}$
satisfies $\bar\rho_6\in T$, as guaranteed by
Lemma~\ref{lem:rho-on-torus} below.
Since $\VV^1_2(U)=\{\bo\}$, we have
$\depth_1(\bar\rho_6)=1$, giving
$b_1(F)=5+\varphi(3)\cdot 1=7$ and
$\Delta_1(t)=(t-1)^5(t^2+t+1)$.
\end{example}

The general fact illustrated by Example~\ref{ex:braid-rho} 
is the content of the following lemma, a special case of
\cite[Thm.~8.3(1)]{PS-plms17}; we include the short proof, since the
lemma is a key ingredient in the proof of Theorem~\ref{thm:general}.

\begin{lemma}
\label{lem:rho-on-torus}
Let $\A$ be a central arrangement of $n=|\A|$ hyperplanes,
and let $\cN$ be a reduced $3$-multinet on $\A$, with associated
admissible map $f_\cN\colon U\to S=\Sigma_{0,3}$ and component
$T_\cN = f_\cN^*(H^1(S;\C^*))\subset\VV^1_1(U)$.
Then $\bar\rho_n=\rho_n^{n/3}\in T_\cN$.
\end{lemma}

\begin{proof}
Let $\rho_S\colon\pi_1(S)\to\C^*$ be the diagonal character
of $S=\Sigma_{0,3}$, sending each meridional loop $c_j$
to $e^{2\pi \ii/3}$. Since $\cN$ is a reduced $3$-multinet of weight $d$,
each $H\in\A$ belongs to exactly one class $\A_j$ ($j=1,2,3$),
and each class has $d$ elements; thus $n=3d$, and the admissible map satisfies
$(f_\cN)_*(\bar x_H)=[c_j]$ in $H_1(S;\Z)$, for $H\in\A_j$. Therefore
\[
  (f_\cN^*\rho_S)(\bar\gamma_H) = \rho_S(c_j) = e^{2\pi \ii/3},
\]
which shows $f_\cN^*\rho_S = \bar\rho_n$.
Hence $\bar\rho_n\in f_\cN^*(H^1(S;\C^*))=T_\cN$.
\end{proof}

\subsection{Formality of the Milnor fiber}
\label{subsec:mf-formal}

Since $F(\A)$ has the homotopy type of a CW-complex of dimension
at most $r-1$, Theorem~\ref{thm:partial-to-formal} gives: $F$ is
$(r-2)$-formal if and only if $F$ is formal. Thus $q$-formality
is a non-trivial condition only for $1\le q\le r-2$.

Applying Lemma~\ref{lem:dp-formality} to the cyclic cover
$\sigma\colon F(\A)\to U(\A)$, and using the formality of
$U(\A)$ \cite{Br}, we obtain:

\begin{corollary}
\label{cor:mf-formal}
If the algebraic monodromy $h_*$ acts trivially on $H_i(F;\Q)$ for 
all $i\le q$, then $F(\A)$ is $q$-formal. In particular, if this holds for all
$i\le r-2$, then $F(\A)$ is formal.
\end{corollary}

Note that if $F$ is simply connected, then $b_1(F)=0$ and
$F$ is automatically $1$-formal. This does not occur for
hyperplane arrangement Milnor fibers (since $b_1(F)\ge n-1\ge 1$
for $n\ge 2$), but it is the situation in the Fern\'andez de
Bobadilla construction discussed in \S\ref{subsec:FdB}.

Since $1$-formality depends only on $\pi_1$, and the slicing
operation of \S\ref{subsec:arr} preserves $\pi_1(F(\A))$
by \S\ref{subsec:mf-background}, the $1$-formality of
$F(\A)$ is likewise preserved. Thus for the purposes of
Question~\ref{q:milnorf} and Theorem~\ref{thm:general},
it suffices to assume $\rank(\A)=3$.

For rank-$3$ arrangements, $F(\A)$ is a
finite $2$-complex, so $1$-formality and formality coincide by
Theorem~\ref{thm:partial-to-formal}. Zuber \cite{Zu} gave the 
first negative answer to Question~\ref{q:milnorf} for the Ceva
arrangement $\A(3,3,3)$; the present paper gives an infinite
family of counterexamples and identifies the combinatorial
mechanism responsible.

\subsection{The Fern\'andez de Bobadilla construction}
\label{subsec:FdB}

Fern\'andez de Bobadilla \cite{FdB} gave a general method for
constructing non-formal Milnor fibers outside the hyperplane
arrangement setting. The key ingredient is a realization theorem:
for any ideal $I=(f_1,\ldots,f_r)\subset\mathcal{O}_{\C^n,O}$,
the Milnor fiber of the function germ
\[
F_I(x,y)=\sum_{i=1}^r y_i f_i(x), \qquad (x,y)\in(\C^{n+r},O),
\]
is homotopy equivalent to the complement $B_\varepsilon\setminus V(I)$
of $V(I)$ in a small ball, with trivial geometric monodromy
\cite[Thm.~1]{FdB}. Applying
this to the coordinate subspace arrangements of Denham--Suciu
\cite{DeS07} in $\C^6$ --- whose complements carry non-trivial
Massey triple products --- produces a family of weighted-homogeneous
polynomials whose Milnor fibers are simply connected and
non-formal.

This construction differs from the approach of the present paper
in two respects. First, the polynomials produced are not defining
polynomials of hyperplane arrangements, living instead in a
higher-dimensional ambient space. Second, the non-formality is
witnessed by non-vanishing Massey products in $H^*(F;\C)$,
whereas our arguments detect the failure of $1$-formality via
the Tangent Cone Theorem and mixed Hodge structure considerations.

This approach was pursued further in \cite{Su-ejm26}. Combining
the realization theorem with the Grbi\'c--Linton constructions of
non-trivial higher Massey products in moment-angle complexes
\cite{GL21}, one obtains, for every $k\ge 1$ and $n\ge 3$, a
weighted-homogeneous polynomial with trivial geometric monodromy
whose Milnor fiber is $(2k+2)$-connected and non-formal, with
non-formality witnessed by a non-trivial $n$-fold Massey product.
For an explicit family of such polynomials, the exact formality
degree of the Milnor fiber is determined, and the Kato--Matsumoto
connectivity bound \cite{KM75} is shown to be sharp. Homogeneous polynomials
with these properties were subsequently constructed by
Ishikawa and Nguyen \cite{IN26}, building on the highly connected
moment-angle manifolds with non-trivial higher Massey products
found by Limonchenko \cite{Li16, Li19}.

\subsection{Aomoto--Betti numbers and monodromy eigenvalues}
\label{subsec:modular}

For a field $\k$, the Orlik--Solomon algebra $A(\A)$ may be
tensored with $\k$ to give a cochain complex
$(A(\A)\otimes\k,\,\delta_{\varsigma})$, where $\delta_{\varsigma}$ 
denotes left-multiplication by the \emph{diagonal class}
$\varsigma=\sum_{H\in\A}e_H\in A^1(\A)$. Introduced 
in \cite{PS-plms17}, the \emph{Aomoto--Betti number} of $\A$ 
at a prime $p$ is the combinatorial invariant
\begin{equation}
\label{eq:beta-p}
\beta_p(\A)\coloneqq\dim_{\FF_p} H^1(A(\A)\otimes\FF_p,\,\delta_{\varsigma}),
\end{equation}
which depends only on $L_{\le 2}(\A)$ and $p$.

The significance of $\beta_p(\A)$ for the Milnor fibration comes
from the \emph{modular bound}
\begin{equation}
\label{eq:beta-e}
e_p(\A) \le \beta_p(\A),
\end{equation}
valid for all primes $p$, where
$e_p(\A)=\depth_1(\rho_n^{n/p})$ is the multiplicity of
the eigenvalue $e^{2\pi \ii/p}$ in $\Delta_1(t)$. This bound was
proved for arrangement complements in \cite[Thm.~1.3]{CO00}, and in
much greater generality, under mild (and necessary) hypotheses, in
\cite[Thm.~C]{PS-tams10}; see also \cite[\S 8.3]{PS-plms17}.
For $p=3$, the bound is sharp under a multiplicity hypothesis:
if $L_2(\A)$ has no flats of multiplicity $3r$ with $r>1$, then
$e_3(\A)=\beta_3(\A)$ \cite[Thm.~1.6]{PS-plms17}. In that case,
$\beta_3(\A)\ne 0$ forces non-trivial monodromy, so that
Corollary~\ref{cor:mf-formal} cannot be used to establish
the formality of $F(\A)$.

Over $\C$, the connection between $\beta_3$ and nets runs deeper:
by \cite[Thm.~1.6]{PS-plms17}, under the same multiplicity
hypothesis, $\beta_3(\A)\ne 0$ if and only if $\A$ supports a
$3$-net, and the number of essential components of
$\RR^1_1(M;\C)$ arising from $3$-nets equals
$(3^{\beta_3(\A)}-1)/2$. Over $\FF_p$, however, the
resonance varieties are less rigid: components need not be
linear \cite{Fa07}, and non-vanishing Massey products in
$H^*(M;\FF_p)$ can appear \cite{Ma07}.

\section{Reflection arrangements}
\label{sect:reflection-arrs}

\subsection{Reflection arrangements}
\label{subsec:reflection}

An important and well-studied class of arrangements arises from
finite reflection groups. Let $W\subset\GL(\C^\ell)$ be a finite
group generated by (pseudo-)reflections. The \emph{reflection
arrangement} of $W$ is the collection $\A(W)$ of all reflecting
hyperplanes of elements of $W$. The complement $M(\A(W))$ is the
space of regular points of $W$ (those with trivial stabilizer),
and its fundamental group is the \emph{pure Artin group}
associated to $W$.

\medskip\noindent
\textbf{Weyl/Coxeter arrangements.}
When $W$ is a finite Weyl or Coxeter group, the arrangement
$\A(W)$ is the complexification of a real hyperplane arrangement.
The primary example is the \emph{braid arrangement} $\A_{n-1}$,
the reflection arrangement of $\mathfrak{S}_n$ acting on $\C^n$,
whose complement $M(\A_{n-1})=\Conf_n(\C)$ is a classifying
space $K(P_n,1)$ for the pure braid group. More generally, for
any finite-type Artin group, the complement of the associated
complexified Coxeter arrangement is a $K(\pi,1)$, by Deligne's
theorem \cite{De72},
providing a direct link between the combinatorics of $L(\A(W))$
and the classifying space topology of the Artin group.
For infinite Coxeter groups, the analogous hyperplane complements
(in the complexified Tits cone) were shown to be aspherical by
Charney and Davis \cite{CD95} for all Coxeter systems satisfying
their (FC) condition, as well as in dimension~$2$, thereby proving
the $K(\pi,1)$ conjecture of Arnold, Pham, and Thom in these cases;
finite $K(\pi,1)$-complexes for the corresponding Artin groups
were constructed in \cite{CD95b}.

\medskip\noindent
\textbf{Complex reflection arrangements.}
More generally, consider finite groups
$W\subset\GL(\C^\ell)$ generated by complex reflections (elements
with a single eigenvalue different from $1$). By the
Shephard--Todd classification \cite{OT}, the irreducible complex
reflection groups form an infinite family $G(m,p,\ell)$ together
with $34$ exceptional groups. The complement $M(\A(W))$ is again
a $K(\pi,1)$-space \cite[Thm.~0.4]{Bessis}, now for the pure
braid group of $W$. For a comprehensive treatment of
reflection arrangements, including an in-depth study of their
combinatorics and topology, we refer to \cite[Ch.~6]{OT}.

\subsection{The monomial family and its Milnor fiber}
\label{subsec:monomial}

Of primary interest to us is the infinite family
$G(m,m,3)\subset\GL(\C^3)$, whose reflection arrangement is the
\emph{monomial arrangement} $\A(m,m,3)$, defined by
\[
  Q(\A(m,m,3))=(z_1^m-z_2^m)(z_1^m-z_3^m)(z_2^m-z_3^m).
\]
It has $n=3m$ hyperplanes. The case $m=1$ gives the braid
arrangement $\A_2$ (type $A_2$, three lines through the origin);
$m=2$ gives the braid arrangement $\A_3$ of
Example~\ref{ex:braid-rho}, with $6$ hyperplanes and $4$ triple
points; and $m=3$ gives the \emph{Ceva arrangement} $\A(3,3,3)$
with $9$ hyperplanes and $12$ triple points, the running example
of this paper.

The intersection lattice of $\A(m,m,3)$ has a uniform
description: there are $3$ flats of multiplicity $m$ and all
remaining flats in $L_2(\A)$ have multiplicity $3$. Each
$\A(m,m,3)$ supports a $(3,m)$-net with partition given by the
three factors of $Q$, with Latin square corresponding to $\Z_m$
\cite[Ex.~5.4]{FY07}. The value of $\beta_3$ depends
on divisibility:
\[
  \beta_3(\A(m,m,3))=
  \begin{cases} 2 & \text{if }3\mid m,\\ 1 & \text{if }3\nmid m,
  \end{cases}
\]
see \cite[Ex.~8.10]{PS-plms17}. The algebraic monodromy of the
full family is determined by the MPP theorem \cite{MPP}:
\begin{equation}
\label{eq:mpp-mono}
  e_p(\A(m,m,3))=\beta_p(\A(m,m,3))
  \quad\text{for all }m\ge 1\text{ and all primes }p.
\end{equation}
Moreover, no monodromy eigenvalues of order other than $3$ occur
\cite{MPP, Di19}: when $3\mid m$, $e_3=2$; when $3\nmid m$, $e_3=1$.

Dimca \cite{Di19} computed $\dim W_1(F)$ for all irreducible
complex reflection arrangements. For the monomial family, using
\eqref{eq:mpp-mono} and the MHS decomposition of \S\ref{subsec:mf-background}:
\begin{itemize}[itemsep=2pt]
  \item if $3\mid m$: $\dim W_1(F)=4$, with
        $\dim H^1(F;\C)_\lambda=2$ for $\lambda=e^{2\pi \ii/3}$;
  \item if $3\nmid m$: $\dim W_1(F)=2$.
\end{itemize}
The case $3\mid m$ is covered by Theorem~\ref{thm:monomial};
for the case $3\nmid m$, see Remark~\ref{rem:no-pincer}.

Among the exceptional groups, $G_{25}\subset\GL(\C^3)$ gives
rise to the \emph{Hessian arrangement} $\A=\A(G_{25})$,
defined by the polynomial
\[
  Q(\A) = z_1 z_2 z_3
  \prod_{j,k\in\Z_3}(z_1+\omega^j z_2+\omega^k z_3),
  \qquad \omega=e^{2\pi \ii/3},
\]
consisting of $12$ hyperplanes with $9$ quadruple points 
and $12$ double points. Dimca shows $\dim W_1(F)=6$; see
\S\ref{subsec:hessian} for a detailed discussion.

\begin{example}
\label{ex:ceva-nets}
The Ceva arrangement $\A(3,3,3)$ is the $m=3$ case of the
monomial family, so $\beta_3=2$ and $e_3=2$. It supports
exactly four $3$-nets, with partitions
$(123|456|789)$, $(147|258|369)$, $(159|267|348)$, and
$(168|249|357)$ in a suitable ordering of the $9$ hyperplanes.
The resonance variety $\RR^1_1(\A(3,3,3))$ has $12$ local
components (from the $12$ triple points) and $4$ essential
$2$-dimensional components, one for each net.
\end{example}

\section{Proof of the main theorems}
\label{sect:main}

We now prove Theorems~\ref{thm:general} and~\ref{thm:monomial}.
Beyond Zuber's original argument for $\A(3,3,3)$, the key point
is that the torsion character $\bar\rho_n$ of
Lemma~\ref{lem:rho-on-torus} lies on the components of
$\VV^1_1(U)$ associated to all reduced $3$-multinets, and so links
any two of them to the Milnor fiber cover, as in the proof of
\cite[Thm.~8.6]{PS-plms17};
see Figure~\ref{fig:pincer} for a schematic of the overall argument.

\subsection{Cup-product vanishing}
\label{subsec:zero-cup}

We start by recalling a lemma from \cite{Zu} (see also the proof
of \cite[Cor.~4.2]{DP-pisa}); we include a proof based on weight
considerations.

\begin{lemma}[Zuber {\cite[Lem.~5]{Zu}}]
\label{lem:cup-vanish}
For all $\alpha\in H^1(F;\C)_\lambda$ and
$\beta\in H^1(F;\C)_{\bar\lambda}$ with $\lambda\ne 1$,
one has $\alpha\cup\beta=0$. Consequently,
$H^1(F;\C)_\lambda\subseteq \RR^1_1(F)$ for every $\lambda\ne 1$.
\end{lemma}

\begin{proof}
Since $h^*\alpha=\lambda\alpha$ and $h^*\beta=\bar\lambda\beta$,
we have
\[
h^*(\alpha\cup\beta)=(h^*\alpha)\cup(h^*\beta)
=\lambda\bar\lambda\,\alpha\cup\beta=\alpha\cup\beta.
\]
Thus $\alpha\cup\beta\in H^2(F;\C)_1\cong H^2(U;\C)$,
which has pure weight $4$. On the other hand, $\alpha$
and $\beta$ have weight $\le 1$, so $\alpha\cup\beta$
has weight $\le 2$. Hence $\alpha\cup\beta=0$.

For the last assertion, let $0\ne\alpha\in H^1(F;\C)_\lambda$.
Since $h^*$ is defined over $\Q$, complex conjugation maps
$H^1(F;\C)_\lambda$ onto $H^1(F;\C)_{\bar\lambda}$. If
$\lambda\ne\bar\lambda$, pick any $0\ne\beta\in H^1(F;\C)_{\bar\lambda}$;
if $\lambda=\bar\lambda=-1$, the space $H^1(F;\C)_{-1}$ is defined
over $\Q$ and carries a pure Hodge structure of weight $1$, hence
has even dimension, and we pick $\beta\in H^1(F;\C)_{-1}$
linearly independent from $\alpha$. In either case,
$\beta\notin\C\alpha$ and $\alpha\cup\beta=0$, so that
$\alpha\in\RR^1_1(F)$.
\end{proof}

\subsection{The lifted pencil}
\label{subsec:lift}

Let $\cN$ be a reduced $3$-multinet on $\A$ of weight $d$, with
classes $\A_1,\A_2,\A_3$. For $j=1,2,3$, set
$C_j=\prod_{H\in\A_j}f_H$, so that $Q=C_1C_2C_3$. As noted in
\S\ref{subsec:nets}, the polynomials $C_1,C_2,C_3$ span a
$2$-dimensional space, so we may write $C_3=a C_1+b C_2$, with
$a,b\in\C^*$. The associated pencil is
$\psi=f_\cN=[C_1:C_2]\colon U\to S=\Sigma_{0,3}$. Let
\[
\hat{S}=\bigl\{(u,v)\in\C^2 \bigm| uv(au+bv)=1\bigr\}
\]
be the Milnor fiber of the pencil $\B$ of $3$ lines in $\C^2$
defined by $uv(au+bv)=0$, so that $S=U(\B)$ and
$\hat{S}=F(\B)=\Sigma_{1,3}$, as in \S\ref{subsec:pencil}.
The projection $\nu\colon\hat{S}\to S$, $(u,v)\mapsto[u:v]$, is the
$\Z_3$-cover classified by the diagonal character
$\hat\rho_3\colon\pi_1(S)\to\Z_3$; its deck group is generated by
the monodromy $g(u,v)=(\omega u,\omega v)$ of $\B$, where
$\omega=e^{2\pi\ii/3}$.

Since $C_1C_2C_3=Q$, the map
$\tilde\psi\colon F\to\hat{S}$, $\tilde\psi(z)=(C_1(z),C_2(z))$,
is well defined, and fits into the commutative diagram
\begin{equation}
\label{eq:lifted-pencil}
\begin{tikzcd}[column sep=32pt, row sep=28pt]
F \ar[d, "\sigma"] \ar[r, "\tilde{\psi}"] & \hat{S} \ar[d, "\nu"] \\
U \ar[r, "\psi"] & S
\end{tikzcd}
\end{equation}
(This is the geometric counterpart of the identity
$f_\cN^*\rho_S=\bar\rho_n$ from Lemma~\ref{lem:rho-on-torus}.)
If, moreover, $\cN$ has a base point of multiplicity one, then
$\tilde\psi$ has connected generic fiber, by \cite[Prop.~2]{Zu}.
Assume this is the case, and set
$E=\tilde{\psi}^*H^1(\hat{S};\C)\subset H^1(F;\C)$.
As $\tilde\psi$ is dominant, $\tilde\psi^*$ is injective in degree $1$,
and so $\dim E=4$, since $b_1(\hat{S})=4$.
The mixed Hodge structure on $H^1(\hat{S};\C)$
(cf.~\S\ref{subsec:pencil} and \cite[Lem.~4]{Zu}) yields
\[
E=E^{1,1}\oplus E^{1,0}\oplus E^{0,1},
\qquad
\dim E^{1,1}=2,\quad
\dim E^{1,0}=1,\quad
\dim E^{0,1}=1,
\]
and in particular
\begin{equation}
\label{eq:E-Hodge}
\dim(E\cap H^{1,0}(F))=1.
\end{equation}
Since $\tilde\psi$ is admissible (as $\chi(\hat{S})=-3<0$), it
follows from Theorem~\ref{thm:arapura} that the subtorus
$\tilde{T}=\tilde\psi^*(H^1(\hat{S};\C^*))$ is a component of
$\VV^1_1(F)$ through $\bo$, with tangent space $E$ at $\bo$.

Finally, $E$ is invariant under the monodromy. Indeed, since $n=3d$
and each $C_j$ is homogeneous of degree $d$, we have
$C_j(e^{2\pi\ii/n}z)=\omega\, C_j(z)$, that is,
$\tilde\psi\circ h=g\circ\tilde\psi$. Hence $h^*\circ\tilde\psi^*
=\tilde\psi^*\circ g^*$, and so $\tilde\psi^*$ maps each eigenspace
of $g^*$ into the corresponding eigenspace of $h^*$. Since the
eigenvalues $\omega$ and $\bar\omega$ of $g^*$ on $H^1(\hat{S};\C)$
each occur with multiplicity one (\S\ref{subsec:pencil}), we
conclude that
\begin{equation}
\label{eq:E-eigen}
E\cap H^1(F;\C)_{\omega}\ne\{0\}
\quad\text{and}\quad
E\cap H^1(F;\C)_{\bar\omega}\ne\{0\}.
\end{equation}

\begin{figure}[t]
\[
\begin{tikzpicture}[baseline=(current bounding box.center),scale=0.9]
\begin{scope}[xshift=-3.8cm]
  \draw[blue, thick, rotate around={-20:(0,0.55)}]
    (0,0.55) ellipse (1.3cm and 0.55cm);
  \draw[red, thick, rotate around={20:(0,0.55)}]
    (0,0.55) ellipse (1.3cm and 0.55cm);
  \node[circle,fill=black,inner sep=1.5pt] (id) at (0,0) {};
  \node at (0.15,0) [label=below:$\mathbf{1}$] {};
  \node[circle,fill=violet,inner sep=2pt] (rho) at (0,1.1) {};
  \node[font=\small] at (-0.55,1.5) {$\bar\rho_n$};
  \draw[violet, ->, bend right=20]
    (0.75,1.62) to (0.12,1.2);
  \node[violet, font=\small, fill=white, inner sep=1pt] at (0.8,1.7) {$\VV^1_2(U)$};
  \node[blue, font=\small] at (-1.6,0.55) {$T_1$};
  \node[red, font=\small] at (1.6,0.55) {$T_2$};
  \node[font=\small] at (0,-1.1) {$\VV^1_1(U)\subset H^1(U;\C^*)$};
\end{scope}

\node at (0,0.4) {$\xrightarrow{\quad\sigma^*\quad}$};

\begin{scope}[xshift=4.5cm]
  \fill[gray!15] (-1.6,-0.9) -- (1.9,-0.9) -- (1.9,1.3) -- (-1.6,1.3) -- cycle;
  \draw[thick] (-1.6,-0.9) -- (1.9,-0.9) -- (1.9,1.3) -- (-1.6,1.3) -- cycle;
  \node[font=\small] at (1.6,-0.5) {$E$};
  \draw[violet, thick] (-1.3,-0.3) -- (1.3,0.65);
  \node[violet, font=\small] at (1.1,.95) {$V$};
  \node[violet, font=\small] at (0.15,-1.3)
    {$\dim(V\cap H^{1,0}(F))\ge 2$};
  \draw[green!50!black, thick] (-0.2,-0.8) -- (-0.2,1.2);
  \node[green!50!black, font=\small] at (-0.2,1.55)
    {$\dim(E\cap H^{1,0}(F))=1$};
  \node[font=\small] at (0.15,-1.85)
    {$\Rightarrow\;\TC_{\bo}(\VV^1_1(F))\ne\RR^1_1(F)$};
\end{scope}
\end{tikzpicture}
\]
\caption{The pincer argument. Left: the components $T_1,T_2$
of $\VV^1_1(U)$ arising from two reduced $3$-multinets meet at the
torsion character $\bar\rho_n$, which therefore lies in
$\VV^1_2(U)$. Right: in $H^1(F;\C)$, this forces
$\dim(V\cap H^{1,0}(F))\ge 2$ for
$V=H^1(F;\C)_\omega\oplus H^1(F;\C)_{\bar\omega}$ (violet), while
the lifted pencil yields a component of $\VV^1_1(F)$ with tangent
space $E$ at $\bo$ and $\dim(E\cap H^{1,0}(F))=1$ (green).
If $\TC_{\bo}(\VV^1_1(F))=\RR^1_1(F)$, then $V\subseteq E$, as
drawn, which is impossible; see Remark~\ref{rem:tc-strict}.}
\label{fig:pincer}
\end{figure}

\subsection{Proof of Theorem~\ref{thm:general}}
\label{subsec:proof-general}

Let $\cN_1$ and $\cN_2$ be two distinct reduced $3$-multinets on $\A$,
where $\cN_1$ has a base point of multiplicity one,
with corresponding components $T_1\ne T_2$ of $\VV^1_1(U)$
(distinct multinets give distinct components, see
\cite[Lem.~6.5]{PS-plms17}).
Since $\A$ supports a reduced $3$-multinet, we have $3\mid n$; set
$\bar\rho_n=\rho_n^{n/3}$, the \emph{reduced diagonal character},
of order $3$, sending each $\bar\gamma_H$ to $e^{2\pi \ii/3}$.
By Lemma~\ref{lem:rho-on-torus}, $\bar\rho_n\in T_1\cap T_2$.
Since $\bar\rho_n$ is a torsion character,
Theorem~\ref{thm:acm} (with $r=s=1$) gives
$\bar\rho_n\in\VV^1_2(U)$.

By formula~\eqref{eq:charpoly}, $e_3(\A)=\depth_1(\bar\rho_n)\ge 2$.
Write $H_\lambda=H^1(F;\C)_\lambda$, and let $\omega=e^{2\pi\ii/3}$.
Then $\dim H_\omega=\dim H_{\bar\omega}=e_3(\A)\ge 2$. The subspace
$V=H_\omega\oplus H_{\bar\omega}$ is the kernel of
$(h^*)^2+h^*+1$, hence is defined over $\Q$; since $h^*$ is an
automorphism of mixed Hodge structures, $V$ is a sub-Hodge structure
of $W_1(F)$, pure of weight $1$. Consequently,
\begin{equation}
\label{eq:h10-dim}
\dim(V\cap H^{1,0}(F)) = \tfrac{1}{2}\dim V = e_3(\A) \ge 2.
\end{equation}

Let $\psi\colon U\to S$ be the pencil associated to $\cN_1$.
By \S\ref{subsec:lift}, the lifted pencil
$\tilde\psi\colon F\to\hat{S}$ gives a $4$-dimensional,
$h^*$-invariant subspace $E\subset H^1(F;\C)$ with
$\dim(E\cap H^{1,0}(F))=1$, by~\eqref{eq:E-Hodge}.

Assume for contradiction that $F$ is $1$-formal.
By Theorem~\ref{thm:tcone}, the irreducible components of
$\RR^1_1(F)$ are precisely the tangent spaces at $\bo$ to
the components of $\VV^1_1(F)$ through $\bo$; in particular,
$E$ is an irreducible component of $\RR^1_1(F)$.
Since $F$ is a smooth quasi-projective variety,
distinct positive-dimensional components of $\VV^1_1(F)$
through $\bo$ meet in finitely many points
\cite[Lem.~6.12]{DPS-duke}; since their tangent spaces at $\bo$
are rationally defined linear subspaces, we conclude that
\begin{equation}
\label{eq:disjoint}
E'\cap E''=\{0\}
\quad\text{for distinct irreducible components $E',E''$ of }\RR^1_1(F).
\end{equation}

By Lemma~\ref{lem:cup-vanish}, the eigenspace $H_\omega$ lies in
$\RR^1_1(F)$. Being a linear subspace, $H_\omega$ is then
contained in a single irreducible component $E'$ of $\RR^1_1(F)$.
By~\eqref{eq:E-eigen}, we have
$E'\cap E\supseteq H_\omega\cap E\ne\{0\}$,
and so $E'=E$, by~\eqref{eq:disjoint}. Thus $H_\omega\subseteq E$,
and, since $E$ is defined over $\Q$, also
$H_{\bar\omega}=\overline{H_\omega}\subseteq E$; that is,
$V\subseteq E$. But then
\[
\dim(V\cap H^{1,0}(F)) \le \dim(E\cap H^{1,0}(F)) = 1,
\]
contradicting~\eqref{eq:h10-dim}. Therefore $F$ is not
$1$-formal. \qed

\begin{remark}
\label{rem:tc-strict}
The argument above uses the $1$-formality of $F$ only through the
equality $\TC_{\bo}(\VV^1_1(F))=\RR^1_1(F)$ provided by
Theorem~\ref{thm:tcone}. Indeed, near $\bo$ the variety
$\VV^1_1(F)$ is a finite union of subtori, so this equality alone
implies that the irreducible components of $\RR^1_1(F)$ are the
tangent spaces at $\bo$ to the components of $\VV^1_1(F)$ through
$\bo$, while~\eqref{eq:disjoint} holds by \cite[Lem.~6.12]{DPS-duke}.
Thus, the proof shows that the inclusion
$\TC_{\bo}(\VV^1_1(F))\subseteq\RR^1_1(F)$ from~\eqref{eq:tc-chain}
is strict, as indicated in Figure~\ref{fig:pincer}; this is the form
in which Zuber states his result for the Ceva arrangement
\cite[Thm.~1]{Zu}.
\end{remark}

\subsection{Purity, non-formality, and the proof strategy}
\label{subsec:proof-remarks}

By Corollary~\ref{cor:mf-formal}, if $F(\A)$ is not $1$-formal,
then the monodromy on $H_1(F;\Q)$ is non-trivial, which in turn forces
$W_1(F)=H^1(F;\C)_{\ne 1}\ne 0$. Thus $H^1(F;\C)$ carries
weights $1$ and $2$ simultaneously, and the mixed Hodge structure
on $H^1(F;\C)$ is not pure of weight $2$.

By Dupont's theorem \cite[Thm.~1.1]{Dp16}, purity of $H^k(X;\C)$
of weight $2k$ for all $k\le q$ is sufficient for the $q$-formality
of a smooth variety $X$; the special case $q=1$ recovers Morgan's
criterion \cite{Morgan}. For the Milnor fiber $F(\A)$ and $q=1$,
Dupont's purity criterion requires $H^1(F;\C)$ to be pure of
weight $2$, i.e., $W_1(F)=0$. Whenever $\A$ supports a reduced
multinet (Proposition~\ref{prop:red-multi}), $W_1(F)\ne 0$ and this
criterion fails, blocking the only available Hodge-theoretic
route to $1$-formality from purity alone.

However, the failure of weight purity on $H^1(F;\C)$ is not by
itself sufficient to conclude non-formality: there exist smooth
quasi-projective varieties with non-pure $H^1$ that are
nonetheless $1$-formal, for instance the Milnor fiber $\Sigma_{1,3}$
of a pencil of $3$ lines (\S\ref{subsec:pencil}), a punctured curve.
What the proof of
Theorem~\ref{thm:general} does is extract finer information from
the Hodge decomposition --- specifically, the constraint
$\dim(E\cap H^{1,0}(F))=1$ from the lifted pencil versus
$\dim(V\cap H^{1,0}(F))\ge 2$ from the ACM depth count --- to derive a
genuine contradiction with $1$-formality via the Tangent Cone
Theorem.

\subsection{Proof of Theorem~\ref{thm:monomial}}
\label{subsec:monomial-family}

For $k=1$, the Ceva arrangement supports four distinct $3$-nets
(Example~\ref{ex:ceva-nets}), so Theorem~\ref{thm:general} applies.
For $k>1$, the arrangement $\A(3k,3k,3)$ has three flats of
multiplicity $3k$, so Theorem~\ref{thm:ps-nets} is not available;
instead, we construct reduced $3$-multinets on $\A(3k,3k,3)$
directly, by pulling back the four $3$-nets on $\A(3,3,3)$ along the
branched cover
\[
  \varphi_k\colon\CP^2\to\CP^2,\qquad
  [z_1:z_2:z_3]\mapsto[z_1^k:z_2^k:z_3^k].
\]
Since $\varphi_k^*(z_i^3-z_j^3)=z_i^{3k}-z_j^{3k}$, the preimage
of the line $\{z_i=\lambda z_j\}$, with $\lambda^3=1$, is the union
of the $k$ lines $\{z_i=\mu z_j\}$ with $\mu^k=\lambda$. Thus,
$\bar\A(3k,3k,3)=\varphi_k^{-1}(\bar\A(3,3,3))$, and each line
$\bar H'$ of $\bar\A(3k,3k,3)$ covers a unique line
$\varphi_k(\bar H')$ of $\bar\A(3,3,3)$.

\begin{lemma}
\label{lem:pullback-nets}
Let $\cN=(\A_1,\A_2,\A_3)$ be a $3$-net on $\A(3,3,3)$, and let
$\A'_j$ be the set of hyperplanes $H'\in\A(3k,3k,3)$ with
$\varphi_k(\bar H')\in\bar\A_j$. Then the partition
$\varphi_k^*\cN=(\A'_1,\A'_2,\A'_3)$, with all multiplicities equal
to $1$, is a reduced $(3,3k)$-multinet on $\A(3k,3k,3)$.
Moreover, if $\cN\ne\cN'$, then $\varphi_k^*\cN\ne\varphi_k^*\cN'$.
\end{lemma}

\begin{proof}
Each class $\A'_j$ consists of the $3k$ lifts of the three lines of
$\bar\A_j$, so condition~\ref{m1} of Definition~\ref{def:multinet} holds
with $d=3k$. Take as base locus $\XX'$ the set of flats
$X'\in L_2(\A(3k,3k,3))$ whose hyperplanes do not all lie in the
same class; condition~\ref{m2} then holds by construction.

A point of $\CP^2$ with exactly one vanishing coordinate lies on
at most one line of $\bar\A(3k,3k,3)$, so the flats of
$\A(3k,3k,3)$ are of two kinds. If all coordinates of
$\bar X'$ are non-zero, then $\varphi_k$ is unramified at
$\bar X'$, and the lines through $\bar X'$ are the lifts through
$\bar X'$ of the lines of $\bar\A(3,3,3)$ through
$\varphi_k(\bar X')$, one lift for each. If $\bar X'$ is a
coordinate point, say $[0:0:1]$, then $\varphi_k(\bar X')=\bar X'$,
and the $3k$ lines through $\bar X'$ are the lifts, $k$ for each, of
the three lines of $\bar\A(3,3,3)$ through $\bar X'$. In either
case, $\varphi_k(\bar X')$ is a triple point of $\bar\A(3,3,3)$,
whose three lines lie either in a single class of the net $\cN$, or
in three different classes. Accordingly, either $X'\notin\XX'$, or
each class $\A'_j$ contains the same number ($1$ or $k$) of
hyperplanes through $X'$; this verifies condition~\ref{m3}.

For condition~\ref{m4}, let $\bar H_1\ne \bar H_2$ be two lines in a
class $\bar\A_j$. Since $\cN$ is a net, their intersection point
$q$ is not in the base locus of $\cN$, so all lines of
$\bar\A(3,3,3)$ through $q$ lie in $\bar\A_j$. As $\bar\A(3,3,3)$
has only triple points, $q$ lies on all three lines of $\bar\A_j$.
If $q$ is a coordinate point, all $3k$ lines of $\bar\A'_j$ pass
through $q$, and $q\notin\XX'$. Otherwise, any lift of $\bar H_1$
meets any lift of $\bar H_2$ at a point over $q$, which is not
in $\XX'$; two lifts of the same line are then connected via a
lift of another line of $\bar\A_j$. In either case, the class
$\A'_j$ is connected in the required sense.

Finally, $\varphi_k^*\cN$ determines $\cN$, since each line of
$\bar\A(3,3,3)$ is the image of its lifts.
\end{proof}

For $k>1$, the pullback of the $3$-net given by the factors of $Q$
is again a $3$-net, while the other three multinets
$\varphi_k^*\cN$ are not nets: each of their classes contains $k$
lines through each coordinate point.

\begin{remark}
\label{rem:cs99}
Lemma~\ref{lem:pullback-nets} shows that, for $k\ge 2$, the arrangement
$\A(3k,3k,3)$ has at least four essential components in $\RR^1_1$,
one for each of the multinets $\varphi_k^*\cN$. This corrects \cite[Lem.~6.3]{CS99},
which asserts that $\A(r,r,3)$ has a single essential component for every
$r\ge 4$; accordingly, in the count of $2$-dimensional components in
\cite[Prop.~6.7]{CS99}, one has $c_q\ge 4$ whenever $3\mid q$.
This is consistent with the equalities $\beta_3=e_3=2$ for these
arrangements \cite[Ex.~8.10]{PS-plms17}, \cite{MPP}.
\end{remark}

\begin{proof}[Proof of Theorem~\ref{thm:monomial}]
Applying Lemma~\ref{lem:pullback-nets} to the four $3$-nets on
$\A(3,3,3)$ from Example~\ref{ex:ceva-nets}, we see that
$\A=\A(3k,3k,3)$ supports four distinct reduced $3$-multinets,
one of which (the pullback of the $3$-net given by the factors of
$Q$) is a $3$-net, and hence has base points of multiplicity one.
Thus, $\A$ satisfies the hypothesis of Theorem~\ref{thm:general},
and therefore $F(\A)$ is not $1$-formal.
\end{proof}

\begin{remark}
\label{rem:no-pincer}
For the monomial arrangement $\A(m,m,3)$ with $3\nmid m$, we have
$\dim W_1(F)=2$ (see \S\ref{subsec:monomial}), and so
$\dim H^{1,0}(F)=1=\dim(E\cap H^{1,0}(F))$ for any lifted pencil.
Thus, the pincer argument yields no contradiction, and the
$1$-formality of $F(\A(m,m,3))$ for $3\nmid m$ and $m\ge 2$
remains open. (For $m=1$, the Milnor fiber is homotopy equivalent
to $\Sigma_{1,3}$, and hence is formal; the case $m=2$ is
Question~\ref{quest:mf-braid}.)
\end{remark}

\section{The \texorpdfstring{$\beta_3$}{beta 3} criterion and open problems}
\label{sect:questions}

The connection between our non-formality criterion and the
mod~$3$ Aomoto--Betti number $\beta_3(\A)$ is particularly
clean, thanks to the following result of Papa\-dima--Suciu
\cite{PS-plms17}.

\begin{theorem}[{\cite[Thm.~1.6]{PS-plms17}}]
\label{thm:ps-nets}
Let $\A$ be an arrangement such that $L_2(\A)$ has no flats of
multiplicity $3r$ for any $r>1$. Then:
\begin{enumerate}
\item \label{ps1}
$\beta_3(\A)\ne 0$ if and only if $\A$ supports a $3$-net.
\item \label{ps2}
$\beta_3(\A)\le 2$.
\item \label{ps3}
$e_3(\A)=\beta_3(\A)$.
\item \label{ps4}
The number of essential $3$-net components of $\VV^1_1(U)$
is $(3^{\beta_3(\A)}-1)/2$.
\end{enumerate}
\end{theorem}

Under the hypotheses of this theorem, the value of $\beta_3(\A)$
completely controls the situation:
\begin{itemize}[itemsep=3pt]
\item If $\beta_3(\A)=0$: there are no $3$-nets, $e_3(\A)=0$,
and the eigenvalue $e^{2\pi \ii/3}$ does not appear in
the monodromy on $H^1(F;\C)$.
\item If $\beta_3(\A)=1$: there is exactly one essential $3$-net component,
giving $e_3(\A)=1$.
Our method does not apply,
since there is no second component to invoke the ACM theorem.
\item If $\beta_3(\A)=2$: there are four essential $3$-net components,
giving $e_3(\A)=2$.
\end{itemize}

In particular, combining the case $\beta_3(\A)=2$ above with 
Corollary~\ref{cor:two-nets}, we obtain:

\begin{corollary}
\label{cor:beta3}
Let $\A$ be an arrangement with $\beta_3(\A)=2$ and no flats of
multiplicity $3r$ for $r>1$. Then $F(\A)$ is not $1$-formal.
\end{corollary}

\begin{remark}
\label{rem:beta3-combinatorial}
The invariant $\beta_3(\A)$ is purely combinatorial: it is the
dimension of the mod~$3$ cocycle space modulo constants in the
Orlik--Solomon algebra, and depends only on $L_{\le 2}(\A)$.
Hence, Corollary~\ref{cor:beta3} gives a combinatorial sufficient
condition for non-$1$-formality of Milnor fibers, computable by
linear algebra over $\FF_3$.
\end{remark}

\subsection{Open questions}
\label{subsec:question}

The results of this paper naturally raise the following questions.

\begin{question}
\label{quest:beta3-1}
Is there an arrangement $\A$ with $\beta_3(\A)=1$ for which
$F(\A)$ is not $1$-formal?
\end{question}

Such an example would require a mechanism for non-formality
beyond the intersection of two essential components. Indeed, if
$\beta_3(\A)=1$, the modular bound~\eqref{eq:beta-e} gives
$e_3(\A)\le 1$, and so
$\dim(V\cap H^{1,0}(F))\le 1=\dim(E\cap H^{1,0}(F))$ for any
lifted pencil; thus the pincer argument yields no contradiction.

\begin{question}
\label{quest:mf-braid}
Are the Milnor fibers of the braid arrangements $\A_n$ formal
for $n\ge 3$? In particular, is the Milnor fiber of $\A_3$
$1$-formal?
\end{question}

For $\A_3$, we have $\Delta_1(t)=(t-1)^5(t^2+t+1)$, so
$\dim W_1(F(\A_3))=2\ne 0$. Corollary~\ref{cor:mf-formal}
does not apply since the monodromy on $H^1(F;\C)$ is
non-trivial, Dupont's purity criterion \cite{Dp16} is
unavailable, and Theorem~\ref{thm:general} does not apply
since $\A_3$ supports only one essential component in
$\VV^1_1(U)$. For $\A_n$ with $n>3$, it was shown in
\cite{MP09} that $\Delta_1(t)=(t-1)^{\binom{n+1}{2}-1}$,
so $F(\A_n)$ is $1$-formal, but whether it is fully formal
remains open.

\begin{question}
\label{quest:beyond-3nets}
Are there arrangements not covered by Theorem~\ref{thm:general}
(e.g., with $\beta_3(\A)=0$) for which $F(\A)$ is nonetheless
not $1$-formal?
\end{question}

Theorem~\ref{thm:ps-nets} establishes $e_3(\A)=\beta_3(\A)$
under the hypothesis that $L_2(\A)$ has no flats of multiplicity
$3r$ for $r>1$. More generally, Papadima--Suciu conjectured
\cite[Conj.~1.9]{PS-plms17} that $e_p(\A)=\beta_p(\A)$ holds
without multiplicity restrictions, for all primes $p$. This was
verified for all complex reflection arrangements \cite{MPP}, but
disproved in general by Yoshinaga \cite{Yo20}: his
icosidodecahedral arrangement has $n=16$ hyperplanes, $15$
quadruple points and $30$ double points, trivial algebraic
monodromy over $\Q$ (so $e_2(\A)=0$), yet $\beta_2(\A)=1\ne 0$.
Thus the conjecture fails at $p=2$, and whether $e_3(\A)=\beta_3(\A)$
holds beyond the multiplicity hypothesis of
Theorem~\ref{thm:ps-nets} remains open.

We also note that no arrangement of rank at least $3$ with
$\beta_p(\A)\ne 0$ for $p>3$ is currently known; by
\cite{MP09}, no such example can arise among subarrangements
of non-exceptional Coxeter arrangements.

\begin{remark}
\label{rem:net-vs-multinet}
The hypothesis of Theorem~\ref{thm:general} requires two
distinct \emph{reduced} $3$-multinets. The proof uses
Lemma~\ref{lem:rho-on-torus}, which identifies $\bar\rho_n$
as a point of $T_\cN$ via the identity $f_\cN^*\rho_S=\bar\rho_n$.
Since $(f_\cN^*\rho_S)(\bar\gamma_H)=e^{2\pi\ii m_H/3}$, this
identity holds whenever all $m_H\equiv 1 \pmod{3}$, and not only
for reduced multinets; cf.~\cite[Ex.~8.11]{PS-plms17}. The other
steps of the proof would have to be re-examined in that setting.
The additional hypothesis that $\cN_1$ has a base point of
multiplicity one is used to ensure that the lifted pencil
$\tilde\psi$ has connected generic fiber \cite[Prop.~2]{Zu}.
Without it, one may still pass to the Stein factorization of
$\tilde\psi$, as in \cite[Thm.~3.1(ii)]{DP-pisa}, but the resulting
curve may have larger genus, and the dimension count in the proof
of Theorem~\ref{thm:general} breaks down.
\end{remark}

\begin{question}
\label{quest:multinet}
Does the conclusion of Theorem~\ref{thm:general} hold under
the weaker hypothesis that $\A$ supports at least two distinct
$3$-multinets with all multiplicities $m_H\equiv 1\pmod{3}$?
\end{question}

\subsection{The Hessian arrangement}
\label{subsec:hessian}

The Hessian arrangement $\A=\A(G_{25})$ is the reflection 
arrangement of the exceptional complex reflection group $G_{25}$
in $\C^3$. It consists of $12$ hyperplanes, whose projectivization 
forms the four completely reducible fibers of the cubic pencil 
in $\CP^2$ generated by $z_1^3+z_2^3+z_3^3$ and $z_1z_2z_3$.
It has $9$ quadruple points and $12$ double points,
$\beta_2(\A)=2$, and $\beta_p(\A)=0$ for $p\ne 2$. Its characteristic polynomial is
$\Delta_1(t)=(t-1)^{11}[(t+1)(t^2+1)]^2$, giving
$\dim W_1(F)=6$ \cite[Ex.~8.7]{PS-plms17}.

The Hessian supports a unique non-trivial $4$-net of weight $3$.
Yuzvinsky conjectured \cite{Yu12} that this is the only
non-trivial $(4,d)$-multinet on any hyperplane arrangement;
this has been verified for all complex reflection arrangements
of rank at least $3$ \cite{MPP}.
By \cite[Thm.~1.5]{PS-plms17}, no arrangement can simultaneously
support essential $3$-net and $4$-net components in $\VV^1_1(U)$;
in particular, $\A(G_{25})$ supports no $3$-nets. More generally,
since $\beta_3(\A)=0$, the modular bound~\eqref{eq:beta-e} gives
$e_3(\A)=0$, and so, by Lemma~\ref{lem:rho-on-torus}, $\A$ supports
no reduced $3$-multinet. Hence Theorem~\ref{thm:general} does not
apply.

\begin{question}
\label{quest:hessian}
Is the Milnor fiber $F(\A(G_{25}))$ of the Hessian arrangement
$1$-formal?
\end{question}

The $4$-net defines an admissible map $\psi\colon U\to\Sigma_{0,4}$,
and the pencil lifts to $\tilde\psi\colon F\to\hat{S}=\Sigma_{3,4}$,
giving a subspace $E=\tilde\psi^*(H^1(\hat{S};\C))$ of dimension
$9$ with MHS decomposition $\dim E^{1,1}=\dim E^{1,0}=\dim E^{0,1}=3$.
So far the Zuber strategy goes through, and the argument
of \S\ref{subsec:proof-general}, applied to each of the eigenspaces
$H^1(F;\C)_\lambda$ with $\lambda\in\{-1,\pm\ii\}$, shows that
$1$-formality of $F$ would force $W_1(F)\subseteq E$. However,
the dimension count gives $\dim H^{1,0}(F)=3=\binom{k-1}{2}
=\dim(E\cap H^{1,0}(F))$ for $k=4$, equality rather than
strict inequality, so the pincer produces no contradiction.
Alternative approaches might include computing Massey products
in $H^*(F;\C)$, or studying the monodromy on $H^2(F;\C)$
\cite[Rem.~5.1]{Di19}.

\subsection{Strong formality}
\label{subsec:strong-formal}

Kohno and Pajitnov \cite{KP} introduced a notion stronger than
formality. For a connected manifold $X$ and a character
$\rho\in\Ch(X)$, let $A^\bullet(X,E_\rho)$ be the complex of
differential forms with values in the flat line bundle $E_\rho$
with monodromy $\rho$. The isomorphisms
$E_\rho\otimes E_{\rho'}\cong E_{\rho\rho'}$ make
\[
\bar{A}^\bullet(X)=\bigoplus_{\rho\in\Ch(X)}A^\bullet(X,E_\rho)
\]
into a cdga, graded by the group $\Ch(X)$, and $X$ is said to be
\emph{strongly formal} if $\bar{A}^\bullet(X)$ is formal as a cdga.
By \cite[Thm.~6.3]{KP}, strong formality forces the degeneration at
the second page of the twisted spectral sequences whose
differentials are the Massey products $\langle\alpha,\dots,\alpha,-\rangle$
on $H^\bullet(X;\C_\rho)$, for all $\rho\in\Ch(X)$ and
$\alpha\in H^1(X;\C)$; compact K\"ahler manifolds are strongly formal
\cite[Thm.~6.5]{KP}.

In work in preparation \cite{Su-Aform}, we show that a cdga graded
by an abelian group, over a field of characteristic zero whose
characters separate the points of the group, is formal if and only
if it is formal as a graded cdga. Consequently, if $X$ is strongly
formal, then, for every subgroup $\DD\subseteq\Ch(X)$, the sub-cdga
$\bigoplus_{\rho\in\DD}A^\bullet(X,E_\rho)$ is formal; in particular,
$X$ itself is formal ($\DD=\{\bo\}$), and so is every connected finite
abelian cover of $X$ ($\DD$ finite).

These results bear directly on the arrangements studied here.
The Milnor fiber $F(\A)$ is a connected, regular $\Z_n$-cover of
$U(\A)$, so if $F(\A)$ is not formal, then $U(\A)$ is not strongly
formal, and neither is $M(\A)\cong U(\A)\times\C^*$, of which $U(\A)$
is a retract \cite{Su-Aform}. By Theorem~\ref{thm:monomial}, this
applies to the monomial arrangements $\A(3k,3k,3)$, for all $k\ge 1$:
their complements are formal, by Brieskorn's theorem \cite{Br}, but
not strongly formal. Strong formality is thus a genuine
strengthening of formality, already among arrangement complements.

For the braid arrangement $\A_3$, Narkawicz \cite{Nark} considered a
pair of characters $\rho_1,\rho_2\in\Ch(U(\A_3))$, obtained by
exponentiating local resonance classes at two of the triple points,
and exhibited a non-vanishing triple Massey product in an
associated twisted cdga \cite[Thm.~6.6.1]{Nark}. By
\cite[Thm.~6.4.1]{Nark}, this cdga is a model for the part of
$\bar{A}^\bullet(U(\A_3))$ lying over the subgroup
$\DD=\langle\rho_1,\rho_2\rangle\cong\Z^2$ of $\Ch(U(\A_3))$; hence,
as observed in \cite{Su-Aform}, $U(\A_3)$ and $M(\A_3)$ are not
strongly formal. Since $\DD$ is
torsion-free, this obstruction is not detected by
any finite cover; in particular, the conclusion is independent of
Question~\ref{quest:mf-braid}.

On the other hand, all the non-$1$-formality results currently known
for Milnor fibers of arrangements --- Zuber's original example and the
infinite family produced here --- rely on the mixed Hodge structure on
$H^*(F(\A);\C)$ as the primary tool. Whether other invariants, such as
Massey products in $H^*(F(\A);\Q)$, can detect the non-formality of
Milnor fibers of arrangements by purely topological means remains a
compelling open problem; see also \cite[Problems~6.3 and 6.4]{Su-ejm26}.


\end{document}